\documentclass{amsart}

\usepackage{tikz}
\usetikzlibrary{decorations.markings}
\usetikzlibrary{shapes}
\usetikzlibrary{backgrounds}

\usepackage{amssymb}
\usepackage[all]{xy}
\usepackage{multirow}

\usepackage{hyperref}
\makeatletter
\let\@@enum@org\@@enum@
\def\@@enum@[#1]{\@@enum@org[\normalfont #1]}
\makeatother

\newtheorem{thm}{Theorem}
\newtheorem{lem}[thm]{Lemma}
\newtheorem{cor}[thm]{Corollary}

\theoremstyle{remark}

\newtheorem*{rem}{Remark}

\theoremstyle{definition}

\begin{document}
\title{The seven-strand braid group is CAT(0)}
\author{Seong Gu Jeong}
\address{Department of Mathematics, Korea Advanced Institute of Science and Technology, Daejeon, 307-701, Korea}
\email{wjdtjdrn@kaist.ac.kr}
\begin{abstract}
We prove that the 7-strand braid group is CAT(0) by elaborating on the argument of Haettel, Kielak and Schwer.
\end{abstract}

\maketitle

\section{Introduction}

In \cite{Char07} Charney asked whether braid group is CAT(0). In \cite{TJ10}, Brady and McCammond proved that braid groups on at most 5 strands are CAT(0). They also classified finite type Artin groups on four Artin generators which is CAT(0) but their proof uses a computer program. In \cite{TDP13}, Haettel, Kielak and Schwer showed that braid group on at most 6 strands are CAT(0) without using computer programs. In fact, both of them proved that the diagonal link of the \textit{Non crossing partition complex} is CAT(1). As proved in \cite{TJ10}[Proposition 8.3], it impiles that braid group is CAT(0). Haettel, Kielak and Schwer used a criterion of Gromov \cite{Grom87}, Bowditch \cite{Bowd95} and Charney-Davis \cite{RM93} to show if non crossing partition complex has no face satisfying certain four conditions, then the diagonal link of the non crossing partition complex is CAT(1). In this thesis, we prove that $B_7$ is CAT(0) group, by proving that for $n=7$, no faces of non crossing partition complex satisfy Haettel, Kielak and Schwer's four conditions with slightly strengthened fourth condition. Our main theorem is 
\begin{thm}\label{thm5} There does not exist a chain $F$ in $\mathcal{NCP}_7$ such that both $F$ and its dual $F^*$ satisfy condition I,II,III and IV. (These conditions are explained in Section 2.7)
\end{thm}
 We will explain the relationship between condition I,II,III and IV and Haettel, Kielak and Schwer's four conditions in section 4. Condition IV is hard to check while conditions I,II,III are easy to check. We will enumerate faces satisfying condition I,II,III and then show no such faces satisfy condition IV using our criterion. In section 3, we will explain the criterion to check condition IV.\\


\section{Preliminaries}

\subsection{Bounded, graded poset} A poset (partially ordered set) is \textit{bounded} if it has maximum and minimum. A bounded poset has \textit{rank} $n$ if every chain is contained in a maximal chain with $n+1$ elements, and a bounded poset is \textit{graded} if its intervals has rank. For an element $x$ of a bounded graded poset, the \textit{rank} of $x$ is the rank of the interval from minimum to $x$.

\subsection{Orthoscheme metric} For a graded poset, its We will introduce a metric on the geometric realization $|P|$ of a bounded graded poset $P$.\\
We first define the $n$-\textit{orthoscheme} $O_n\subset\mathbb{R}^n$ be the $n$-simplex in $\mathbb{R}^n$ with vertex set $\{(0,0,0,\cdots,0),(1,0,0,\cdots,0),(1,1,0,\cdots,0),\cdots,(1,1,\cdots,1)\}$. We will define a metric called \textit{orthoscheme metric} on $|P|$ by letting any $n$-simplex $\{g_0,\cdots,g_n\}$ (indices correspond to their ranks) is isometric to $O_n$ via the isometry sending $g_0$ to $(0,\cdots,0)$, $g_1$ to $(1,0,\cdots,0)$,$\cdots$,$g_n$ to $(1,\cdots,1)$ and endow the induced length metric. From now on, we will endow orthoscheme metric for every geometric realization of a bounded graded poset.

\subsection{Non crossing partition complex}[]
Consider the set of all partitions of the set $\{1,\cdots,n\}$ that is partially ordered by declaring $P\le Q$ for partitions $P,Q$ if for any $\sigma\in P$, $\sigma\subseteq\tau$ for some $\tau\in Q$. Then $\{\{1,2,\cdots,n\}\}$ is the unique maximal element and $\{\{1\},\cdots,\{n\}\}$ is the unique minimal element.
Let $\mathcal P_n$ be the set of all partitions of $\{1,\cdots,n\}$ that are not neither $\{\{1,2,\cdots,n\}\}$ nor $\{\{1\},\cdots,\{n\}\}$.

Sometimes it is convenient to identify $\{1,2,\cdots,n\}$ with $U_n:=\{e^{2ki\pi/n}|1\leq k\leq n\}\subset\mathbb{C}$. For $P\in\mathcal P_n$, $\sigma\in P$ is called a \emph{block} if $\sigma$ has more than one element. We note that a partition is uniquely determined by its blocks via adding singletons. A block $\sigma$ is \emph{consecutive} if $\sigma$ consists of consecutive integers modulo $n$, that is, they are consecutive in $U_n$. From now on, we assume that all computations on $\{1,2,\cdots,n\}$ are done modulo $n$.

Blocks $\sigma$ and $\tau$ are \emph{crossing} if $\sigma\cap\tau=\emptyset$ and there are $\sigma_1,\sigma_2\in\sigma$ and $\tau_1,\tau_2\in\tau$ such that either $\sigma_1<\tau_1<\sigma_2<\tau_2$ or $\tau_1<\sigma_1<\tau_2<\sigma_2$. For $P\in \mathcal P_n$, $P$ is \emph{non-crossing}, if no two blocks of $P$ are crossing.  Let $\mathcal{NCP}_n$ be the subposet of $(\mathcal P_n,\le)$ consisting of all non-crossing partitions. For $P, Q\in \mathcal{NCP}_n$, $P$ and $Q$ are \emph{crossing}, written $P\nparallel Q$, if a block of $P$ and a block of $Q$ are crossing. Otherwise $P$ and $Q$ are \emph{non-crossing}, written $P\parallel Q$.

For later use, we will denote $NCP_n$ to be the poset of all non crossing partitions (with maximum and minimum), then it is bounded and graded. We will call $|NCP_n|$ \textit{non crossing partition complex}.

\subsection{Brady complex}
We will define a complex that is relevant to braid group and is isometric to non crossing partition complex.\\
Let $B_n$ be the braid group on $n$-strands. We will use the band generator presentation for $B_n$. We will define the \textit{Brady complex} $K_n$ of $B_n$. $K_n$ is the geometric realization of the bounded graded poset with vertex set $\{g\in B_n|id\leq g\leq \delta\}$ with the prefix order on positive braids. $B_n$ acts on $K_n$ simplicially by $g\cdot\{g_0,\cdots,g_k\}=\{gg_0,\cdots,gg_k\}$. It is proved in \cite{Brad01} that the quotient complex $K_n/B_n$ is compact and a $K(B_n,1)$ space. Therefore, to prove $B_n$ is CAT(0), it suffices to prove that $K_n/B_n$ is locally CAT(0), because $B_n$ acts properly discontinuously cocompactly on the universal cover of $K_n/B_n$. It is proved in \cite{TJ10}[Definition 8.2] that $K_n/B_n$ is as a metric space, splits as a product of a PE-complex $Y$ and the circle of lenth $\sqrt{n}$ and that, the link of the unique vertex in $Y$ is isometric to the diagonal link $\text{lk}(e_{01},K_n)$ of $K_n$, where $e_{01}$ is the edge in $K_n$ connecting $id$ and $\delta$. Hence, if $\text{lk}(e_{01},K_n)$ is CAT(1), then $Y$ is locally CAT(0) and so $K_n/B_n$ is locally CAT(0) since it is a product of two locally CAT(0) spaces.\\

\subsection{Rank and corank of a chain} For a chain $F=\{P_1,\cdots,P_k\}$ of $\mathcal{NCP}_n$, its \emph{rank} is defined by $\text{rk} F=\{\text{rk} P_1,\cdots,\text{rk} P_k\}$ and its \emph{corank} is defined by $\text{cork} F=\{1,\cdots,n-2\}\setminus\text{rk} F$.

\subsection{Duality}

In this subsection, we will explain the duality of $\mathcal{NCP}_n$, which will be used to reduce the number of cases that have to be checked in the main theorem.\\
For $P\in\mathcal{NCP}_n$, its \emph{dual} $P^*$ is the partition containing all blocks $\sigma=\{i_1,\cdots,i_k\}$ written clockwise on $U_n$ such that there exist distinct members $\sigma_1,\cdots,\sigma_k\in P$ such that $i_1\in\sigma_1,\cdots,i_k\in\sigma_k$ and $i_1+1\in\sigma_2,i_2+1\in\sigma_3,\cdots,i_{k-1}+1\in\sigma_k,i_k+1\in\sigma_1$. It is easy to check that $P^*\in\mathcal{NCP}_n$ and $P\le Q$ iff $Q^*\le P^*$. Note that $\text{rk}P^*=n-1-\text{rk}P$. 

\subsection{The four conditions} We will introduce four conditions (which are variation of the conditions used by Haettel, Kielak and Schwer) on a chain of $\mathcal{NCP}_n$. It turns out that if no chains of $\mathcal{NCP}_n$ satisfied the four conditions, then $B_n$ is CAT(0) group.\\
Let $F$ be a chain in the poset $\mathcal{NCP}_n$. For $1\le i \le n$, we define $F_i$ to be the smallest subset of $\{1,2,\ldots,n\}$ that appears as a block of a partition in $F$ containing $i$, or $F_i=\{1,2,\cdots,n\}$ if the singleton $\{i\}$ is a member of all partitions in $F$. We remark that for a chain $F$ in $\mathcal{NCP}_n$ and $1\leq i,j\leq n$, $F_j\subseteq F_i$ iff $j\in F_i$. Indeed, if $F_j\subseteq F_i$, then $j\in F_j\subseteq F_i$ by definition. If $j\in F_i$ then $F_i$ contains both $i$ and $j$ and $F_j$ is the smallest block of a partition in $F$ containing $j$ and so $F_j\subseteq F_i$. Set $F_i'=F_i\cap\{i-1,i+1\}$.

For a chain $F$ in $\mathcal{NCP}_n$, we list four conditions that $F$ may satisfy:
\begin{enumerate}
\item[I.] $\text{cork} F$ contains consecutive integers.
\item[II.] There exists an element $P\in F$ such that $P$ is not of the form: $P$ have exactly one block that is consecutive.
\item[III.] Let $F=\{P_1,\cdots,P_k\}$ so that $i<j$ if $P_i<P_j$. There are $P^+,P^-\in \mathcal{NCP}_n$ satisfying:
    \begin{enumerate}
    \item[(i)] either $P_i<P^+,P^-<P_{i+1}$ for some $1\leq i\leq k-1$, or $P^+,P^-<P_1$, or $P_k<P^+,P^-$;
    \item[(ii)] $P^+\nparallel P^-$.
    \end{enumerate}
\item[IV.] There is a maximal chain $C$ in $\mathcal{NCP}_n$ such that $F_i'\cap C_i'=\emptyset$ for all $1\leq i \leq n$.
\end{enumerate}

\section{Proof of main theorem}

We will prove the main theorem: There does not exist a chain $F$ in $\mathcal{NCP}_7$ such that both $F$ and its dual $F^*$ satisfy condition I,II,III and IV.\\

Note that, for given chain $F$, we can check condition I and II easily. Condition III can be checked in the following way. We write $F=\{P_1,\cdots,P_k\}$, where $i<j$ iff $\text{rk}P_i<\text{rk}P_j$ and check whether
\begin{enumerate}
\item[i)] $P_k$ has at least 4 distinct members $\sigma_1,\sigma_2,\sigma_3,\sigma_4$ such that $\sigma_1\cup\sigma_3$ and $\sigma_2\cup\sigma_4$ are crossing or
\item[ii)] For some $1\leq i \leq k-1$, $P_i$ has at least 4 distinct members $\sigma_1,\sigma_2,\sigma_3,\sigma_4$ as a subset of a single block $\sigma\in P_{i+1}$ such that $\sigma_1\cup\sigma_3$ and $\sigma_2\cup\sigma_4$ are crossing or
\item[iii)] $P_1$ has a block consists of at least 4 members.
\end{enumerate}
Hence we can enumerate chains satisfying condition I,II and III up to symmetries of $U_7$ and duality (Note that condition I,II,III and IV are invariant under symmetries of $U_7$). Then we will show none of them satisfies condition IV.
But condition IV is not easy to check. Our strategy is to check condition IV in the following way: Given a chain $F$, we can calculate $F_i'$ for $1\leq i\leq n$. These $F_i'$ are obstructions to construct a maximal chain $C$ satisfying $F_i'\cap C_i'=\emptyset$ for all $1\leq i \leq n$. We will give some conditions of configurations of $\{F_1,\cdots,F_n\}$ that make $C_i$ cannot contain some elements (Lemma \ref{lem3}). Let $F$ be a chain in $\mathcal{NCP}_n$ and $1\le i\le n$.\\

We say that \textit{$F$ satisfies $(1,+,i)$} if
\begin{enumerate}
\item[] $F_i\supseteq F_{i+1}$.
\end{enumerate}

We say that \textit{$F$ satisfies $(2,+,i)$} if either
\begin{enumerate}
\item $F_i=F_{i+1}$ or
\item $F_i\supseteq F_{i+1}\supseteq F_{i+2}$
\end{enumerate}

We say that \textit{$F$ satisfies $(3,+,i)$} if either
\begin{enumerate}
\item $F_i=F_{i+1}\supseteq F_{i+2}$ or
\item $F_i=F_{i+1}\subseteq F_{i+2}$ or
\item $F_i\supseteq F_{i+1}=F_{i+2}$ or
\item $F_i\supseteq F_{i+1}\supseteq F_{i+2}\supseteq F_{i+3}$.
\end{enumerate}

we say that \textit{$F$ satisfies $(4,+,i)$} if either
\begin{enumerate}
\item $F_i=F_{i+1}=F_{i+2}$ or
\item $F_i=F_{i+1}\supseteq F_{i+2}\supseteq F_{i+3}$ or
\item $F_i=F_{i+1}\subseteq F_{i+2}\supseteq F_{i+3}$ or
\item $F_i=F_{i+1}\subseteq F_{i+2}\subseteq F_{i+3}$ or
\item $F_i\supseteq F_{i+1}=F_{i+2}\supseteq F_{i+3}$ or
\item $F_i\supseteq F_{i+1}=F_{i+2}\subseteq F_{i+3}$ or
\item $F_i\supseteq F_{i+1}\supseteq F_{i+2}=F_{i+3}$ or
\item $F_i\supseteq F_{i+1}\supseteq F_{i+2}\supseteq F_{i+3}\supseteq F_{i+4}$.
\end{enumerate}

We will also say \textit{$F$ satisfies $(k,+,i)(m)$} if it satisfies condition (m) among conditions for $(k,+,i)$.
By replacing $+$ sign by $-$ sign in above conditions, we can define \textit{$F$ satisfies $(k,-,i)$} and \textit{$F$ satisfies $(k,-,i)(m)$}. For example,
we say that $(F,i)$ satisfies $(2,-,i)$ if either
\begin{enumerate}
\item $F_i=F_{i-1}$ or
\item $F_i\supseteq F_{i-1}\supseteq F_{i-2}$.
\end{enumerate}
We say that $F$ satisfies $(4,-,i)(6)$ if
\begin{enumerate}
\item[] $F_i\supseteq F_{i-1}=F_{i-2}\subseteq F_{i-3}$.
\end{enumerate}

The following lemma and corollary are useful to prove Lemma \ref{lem3}.

\begin{lem}\label{lem1} Let $F$ be a chain in $\mathcal{NCP}_n$ and $1\leq i,j\leq n$. If $F_i\nsubseteq F_j$ and $F_j\nsubseteq F_i$, then $F_i\cap F_j=\emptyset$ and $F_i\parallel F_j$.
\end{lem}

\begin{proof} Suppose that $F_i\nsubseteq F_j$ and $F_j\nsubseteq F_i$. Then none of them are equal to $\{1,2,\cdots,n\}$, so $F_i$ is a block of $P_1\in F$ and $F_j$ is a block of $P_2\in F$. Without loss of generality we may assume $P_1\leq P_2$, since $F$ is a chain. Hence, by definition, $F_i\subseteq^\exists\sigma\in P_2$. Since we assumed $F_i\nsubseteq F_j$, we have $\sigma\neq F_j$ (hence $\sigma\cap F_j=\emptyset$). So $F_i\cap F_j\subseteq\sigma\cap F_j=\emptyset$. To prove $F_i\parallel F_j$, suppose that $F_i\nparallel F_j$, then $\sigma\nparallel F_j$, a contradiction occurs.
\end{proof}

\begin{cor}\label{cor2} Let $F$ be a chain in $\mathcal{NCP}_n$, $1\leq i\leq n$, $2\leq k\leq n-1$ and $i+1\leq j\leq i+k-1$. If $i,i+k\in F_i$ and $j\notin F_i$, then either $F_i\subseteq F_j$ or $F_j\subseteq\{i+1,\cdots,i+k-1\}$.
\end{cor}

\begin{proof} Suppose $i,i+k\in F_i$ and $j\notin F_i$, then for each $i+1\leq j\leq i+k-1$, since $F_j\nsubseteq F_i$, we have $F_i\subseteq F_j$ or ($F_i\cap F_j=\emptyset$ and $F_i\parallel F_j$) by Lemma \ref{lem1}. The latter case implies $F_j\subseteq\{i+1,\cdots,i+k-1\}$ as desired.
\end{proof}

\begin{lem}\label{lem3} Let $F$ be a chain in $\mathcal{NCP}_n$. If $F$ satisfies $(k,+,i)$ ($(k,-,i)$, respectively) for some $1\leq i\leq n$, then for all maximal chain $C$ satisfying $F_j'\cap C_j'=\emptyset$ for all $1\leq j\leq n$, we have $i+1,\cdots,i+k\notin C_i$ ($i-k,\cdots,i-1\notin C_i$, respectively).
\end{lem}

\begin{proof} Clearly, it suffices to prove for $(k,+,i)$ cases only. Observe that if $F$ satisfies $(k,+,i)$, then $F$ satisfies $(k,+,j)$ for all $1\leq j \leq k$. Choose a maximal chain $C$ satisfying $F_j'\cap C_j'=\emptyset$ for all $1\leq j\leq n$.
If $F$ satisfies $(1,+,i)$, then $i+1\in F_i$ and $F_i'\cap C_i'=\emptyset$, so $i+1\notin C_i$.\\
\\
If $F$ satisfies $(2,+,i)$, then we have $i+1\notin C_i$. Assume that $i+2\in C_i$, then by Corollary \ref{cor2}, $C_{i}\subseteq C_{i+1}$ and hence $i,i+2\in C_{i+1}'$. But $i\in F_{i+1}'$ or $i+2\in F_{i+1}'$, contradiction.\\
\\
If $F$ satisfies $(3,+,i)$, then we have $i+1,i+2\notin C_i$. Assume that $i+3\in C_i$ then By Corollary \ref{cor2}, either $i,i+3\in C_{i+1}$ or $C_{i+1}=\{i+1,i+2\}$ (So condition (1) is discarded). Suppose that $i,i+3\in C_{i+1}$, then $i+2\notin C_{i+1}$ since $F_{i+1}'$ is nonempty. Since $i+1,i+3\in C_{i+1}$ and $i+2\notin C_{i+1}$, by Corollary \ref{cor2} we have $i+1,i+3\in C_{i+2}'$ and a contradiction occurs because $F_{i+2}'$ is nonempty in conditions (2),(3) and (4). Suppose that $C_{i+1}=\{i+1,i+2\}$, then since there is a block $\{i+1,i+2\}$ of an element of $C$, we have $C_{i+2}=\{i+1,i+2\}$ and a contradiction occurs since $i+2\in F_{i+1}'$ or $i+1\in F_{i+2}'$ in all cases.\\
\\
If $F$ satisfies $(4,+,i)$, then we have $i+1,i+2,i+3\notin C_i$. Assume that $i+4\in C_i$, then by Corollary \ref{cor2}, either $C_{i}\subseteq C_{i+1}$ or $C_{i+1}\subseteq \{i+1,i+2,i+3\}$.\\

\textbf{Claim)} $C_{i}\nsubseteq C_{i+1}$\\
If $C_{i}\subseteq C_{i+1}$, then $i,i+4\in C_{i+1}$ and conditions (1),(2),(3) and (4) are discarded. Since $F_{i+1}'\neq\emptyset$ for all cases, we have $i+2\notin C_{i+1}$. Hence, $i+1,i+4\in C_{i+2}$ or $C_{i+2}=\{i+2,i+3\}$. Suppose $i+1,i+4\in C_{i+2}$ (conditions (5) and (6) are discarded and conditions (7) and (8) remain.). Since $F_{i+2}'\neq\emptyset$ in all cases, $i+3\notin C_{i+2}$. Hence, $i+2,i+4\in C_{i+3}'$ by Corollary \ref{cor2} and $F_{i+3}'\neq\emptyset$ for conditions (7) and (8), a contradiction occurs. Suppose $C_{i+2}=\{i+2,i+3\}$, then $C_{i+3}=\{i+2,i+3\}$ and a contradiction occurs because $i+3\in F_{i+2}'$ or $i+2\in F_{i+3}'$ in conditions (5),(6),(7) and (8). So claim is proved.\\

Hence $C_{i+1}\subseteq \{i+1,i+2,i+3\}$. Which means that either $C_{i+1}=\{i+1,i+2\}$ or $C_{i+1}=\{i+1,i+3\}$ or $C_{i+1}=\{i+1,i+2,i+3\}$. Suppose that $C_{i+1}=\{i+1,i+2\}$, then $C_{i+2}=\{i+1,i+2\}$ and a contradiction occurs since $i+2\in F_{i+1}'$ or $i+1\in F_{i+2}'$ in all cases. Suppose that $C_{i+1}=\{i+1,i+3\}$, then by Corollary \ref{cor2}, we have $i+1,i+3\in C_{i+2}'$ and a contradiction occurs since $F_{i+2}'\neq\emptyset$ in all cases. So we have $C_{i+1}=\{i+1,i+2,i+3\}$ and then $\{i+1,i+2,i+3\}$ is a block of some partition $P\in C$. Since $C$ is a maximal chain, there is a partition $P'\in C$ such that $P'<P$ and $P'$ has a block $\sigma$ properly contained in $\{i+1,i+2,i+3\}$. We know that $i+1\notin\sigma$ by definition of $C_{i+1}$. So $\sigma=\{i+2,i+3\}$ and hence $C_{i+2}=\{i+2,i+3\}$ and $C_{i+3}=\{i+2,i+3\}$. A contradiction occurs since  $i+2\in F_{i+1}'$ or $i+3\in F_{i+2}'$ or $i+2\in F_{i+3}'$ in all cases.\\
\end{proof}

\begin{cor}\label{cor4} Let $F$ be a chain in $\mathcal{NCP}_7$. If for some $1\leq i\leq 7$, $F$ satisfies condition $(k,+,i)$ and $(l,-,i)$ and $k+l\geq 6$, then $F$ does not satisfy condition IV.
\end{cor}

\begin{proof} If $F$ satisfy condition IV, we can choose a maximal chain $C$ such that for all $1\leq i \leq n$, $F_i'\cap C_i'=\emptyset$. By Lemma \ref{lem3}, $i-l,\cdots,i-1\notin C_i$ and $i+1,\cdots,i+k\notin C_i$. Since $k+l\geq6$, we have $C_i\subseteq\{i\}$, which is a contradiction because $C_i$ must contain $i$ properly.
\end{proof}

Now we are ready to give a proof of Theorem \ref{thm5}.

\begin{proof}[Proof of theorem \ref{thm5}] We will enumerate all chains satisfying condition I,II and III up to duality, then show none of them satisfying condition IV.\\
Chains satisfying condition I has rank up to duality :\\
$\{1\},\{2\},\{3\},\{1,2\},\{1,3\},\{1,4\},\{1,5\},\{2,3\},\{1,2,3\},\{1,2,5\}$.\\
Note that condition II,III and IV are invariant under symmetries of $U_7$. So we enumerate chains satisfying condition I,II and III up to symmetries of $U_7$. For convenience, we will use a simplified notation for $F$. We will omit non block elements of partitions and will write block in a simplified form. For example, we write $F=\{\{\{1,2\},\{3\},\{4\},\{5\},\{6\},\{7\}\},\{\{1,2\},\{3\},\{4,5,6\},\{7\}\}\}$ as $F=\{\{12\},\{12,456\}\}$.\\
Also, we will omit cases which have the same set of $F_i'$'s as a case previously appeared because condition IV is a condition for $F_i'$'s only. For example, i omitted $\{\{12\},\{12,34\}\}$ in $\text{rk} F=\{1,2\}$ case because it has the same set of $F_i'$'s as the case $\{\{12,34\}\}$ in $\text{rk} F=\{2\}$ case.\\
Checking condition II is straightforward. To check condition III, we write $F=\{P_1,\cdots,P_k\}$, where $i<j$ iff $\text{rk}P_i<\text{rk}P_j$ and check whether
\begin{enumerate}
\item[i)] $P_k$ has at least 4 distinct members $\sigma_1,\sigma_2,\sigma_3,\sigma_4$ such that $\sigma_1\cup\sigma_3$ and $\sigma_2\cup\sigma_4$ are crossing or
\item[ii)] For some $1\leq i \leq k-1$, $P_i$ has at least 4 distinct members $\sigma_1,\sigma_2,\sigma_3,\sigma_4$ as a subset of a single block $\sigma\in P_{i+1}$ such that $\sigma_1\cup\sigma_3$ and $\sigma_2\cup\sigma_4$ are crossing or
\item[iii)] $P_1$ has a block consists of at least 4 members.
\end{enumerate}

Here is an enumeration of chains satisfying condition I,II and III up to duality and symmetries of $U_7$ (This can be done by inspection).\\
\begin{itemize}
\item[i)]
$\text{rk} F=\{1\}$\\
$F=\{\{13\}\}$, $\{\{14\}\}$
\item[ii)]
$\text{rk} F=\{2\}$\\
$F=\{\{124\}\}$, $\{\{12,34\}\}$, $\{\{12,35\}\}$, $\{\{12,37\}\}$, $\{\{12,45\}\}$, $\{\{12,46\}\}$
\item[iii)] 
$\text{rk} F=\{3\}$\\
$F=\{\{1235\}\}$, $\{\{1245\}\}$, $\{\{1246\}\}$, $\{\{123,45\}\}$, $\{\{123,46\}\}$, $\{\{12,34,56\}\}$
\item[iv)]
$\text{rk} F=\{1,2\}$\\
$F=\{\{13\},\{123\}\}$, $\{\{12\},\{124\}\}$, $\{\{14\},\{124\}\}$, $\{\{24\},\{124\}\}$
\item[v)] 
$\text{rk} F=\{1,3\}$\\
$F=\{\{13\},\{1234\}\}$, $\{\{14\},\{1234\}\}$, $\{\{12\},\{123,45\}\}$, $\{\{13\},\{123,45\}\}$,\\
$\{\{23\},\{123,45\}\}$, $\{\{12\},\{123,56\}\},\{\{13\},\{123,56\}\}$
\item[vi)] 
$\text{rk} F=\{1,4\}$\\
$F=\{\{15\},\{12345\}\}$, $\{\{12\},\{12346\}\}$, $\{\{16\},\{12346\}\}$, $\{\{23\},\{12346\}\}$,\\
$\{\{12\},\{12356\}\}$, $\{\{16\},\{12356\}\}$, $\{\{56\},\{12356\}\}$, $\{\{56\},\{1234,56\}\}$,\\
$\{\{57\},\{1234,57\}\}$, $\{\{67\},\{1235,67\}\}$, $\{\{67\},\{1245,67\}\}$
\item[vii)]
$\text{rk} F=\{1,5\}$\\
$F=\{\{13\},\{123456\}\}$, $\{\{16\},\{123456\}\}$, $\{\{24\},\{123456\}\}$, $\{\{26\},\{123456\}\}$,\\
$\{\{12\},\{12345,67\}\}$, $\{\{15\},\{12345,67\}\}$, $\{\{23\},\{12345,67\}\}$, $\{\{67\},\{12345,67\}\}$,\\
$\{\{56\},\{1234,567\}\}$, $\{\{57\},\{1234,567\}\}$
\item[viii)]
$\text{rk} F=\{2,3\}$\\
$F=\{\{124\},\{1234\}\}$
\item[ix)]
$\text{rk} F=\{1,2,3\}$\\
$F=\{\{13\},\{123\},\{1234\}\}$, $\{\{14\},\{124\},\{1234\}\}$, $\{\{24\},\{124\},\{1234\}\}$
\item[x)] 
$\text{rk} F=\{1,2,5\}$\\
$F=\{\{12\},\{126\},\{123456\}\}$, $\{\{16\},\{126\},\{123456\}\}$, $\{\{12\},\{12,34\},\{123456\}\}$,\\
$\{\{12\},\{12,45\},\{123456\}\}$, $\{\{12\},\{12,56\},\{123456\}\}$, $\{\{16\},\{16,23\},\{123456\}\}$,\\
$\{\{16\},\{16,34\},\{123456\}\}$, $\{\{23\},\{23,45\},\{123456\}\}$
\end{itemize}

Now we will prove that all of these chains does not satisfy condition IV by using Lemma \ref{lem3} and Corollary \ref{cor4}. By contradiction, suppose that for each given $F$, let $C$ be a maximal chain satisfying $F_i'\cap C_i'=\emptyset$ for all $1\leq i \leq n$.
\begin{itemize}
\item[(1)]
$F=\{\{13\}\}$, $\{\{14\}\}$, $\{\{124\}\}$ or $\{\{12,34\}\}$\\
$F$ satisfies $(3,+,6)(1)$ and $(3,-,6)(1)$ (and Corollary \ref{cor4} applies).
\item[(2)]
$F=\{\{12,35\}\}$\\
$F$ satisfies $(3,+,7)(3)$ and $(3,-,7)(1)$.
\item[(3)]
$F=\{\{12,37\}\}$\\
$F$ satisfies $(3,+,5)(1)$ and $(3,-,5)(1)$.
\item[(4)]
$F=\{\{12,45\}\}$\\
$F$ satisfies $(3,+,7)(3)$ and $(3,-,7)(1)$.
\item[(5)]
$F=\{\{12,46\}\}$\\
$F$ satisfies $(1,+,3)(1)$ and $(4,-,3)(6)$, so $C_3=\{3,5\}$ (since by Lemma \ref{lem3}, $(1,+,3)(1)$ implies $4\notin C_3$ and $(4,-,3)(6)$ implies $6,7,1,2\notin C_3$ and $C_3$ properly contains 3) and hence $C_5=\{3,5\}$ (since by definition, if $C_3=\{3,5\}$, then $C_5=\{3,5\}$) but $F$ satisfies $(4,+,7)(6)$ and $(1,-,7)(1)$, so $C_7=\{5,7\}$ and hence $C_5=\{5,7\}=\{3,5\}$, a contradiction.
\item[(6)]
$F=\{\{1235\}\}$ or $\{\{1245\}\}$\\
$F$ satisfies $(3,+,7)(3)$ and $(3,-,7)(1)$
\item[(7)]
$F=\{\{1246\}\}$\\
$F$ satisfies $(1,+,3)(1)$ and $(4,-,3)(6)$, so $C_3=\{3,5\}$ and hence $C_5=\{3,5\}$ but $F$ satisfies $(4,+,7)(6)$ and $(1,-,7)(1)$, so $C_7=\{5,7\}$ and hence $C_5=\{5,7\}=\{3,5\}$, a contradiction.
\item[(8)]
$F=\{\{123,45\}\}$\\
$F$ satisfies $(3,+,7)(3)$ and $(3,-,7)(1)$
\item[(9)]
$F=\{\{123,46\}\}$ or $\{\{12,34,56\}\}$\\
$F$ satisfies $(3,+,7)(3)$ and $(3,-,7)(3)$
\item[(10)]
$F=\{\{13\},\{123\}\}$, $\{\{12\},\{124\}\}$, $\{\{14\},\{124\}\}$, $\{\{24\},\{124\}\}$, $\{\{13\},\{1234\}\}$\\
or $\{\{14\},\{1234\}\}$\\
$F$ satisfies $(3,+,6)(1)$ and $(3,-,6)(1)$.
\item[(11)]
$F=\{\{12\},\{123,45\}\}$, $\{\{13\},\{123,45\}\}$ or $\{\{23\},\{123,45\}\}$\\
$F$ satisfies $(3,+,6)(1)$ and $(3,-,6)(3)$.
\item[(12)]
$F=\{\{12\},\{123,56\}\}$ or $\{\{13\},\{123,56\}\}$\\
$F$ satisfies $(4,+,4)(6)$ and $(1,-,4)(1)$ and so $C_4=\{2,4\}$ and hence $C_2=\{2,4\}$ but $F$ satisfies $(1,+,7)(1)$ and $(4,-,7)(6)$, so $C_7=\{2,7\}$ and hence $C_2=\{2,7\}=\{2,4\}$, a contradiction.
\item[(13)]
$F=\{\{15\},\{12345\}\}$\\
$F$ satisfies $(3,+,3)(1)$ and $(3,-,3)(1)$.
\item[(14)]
$F=\{\{12\},\{12346\}\}$\\
$F$ satisfies $(2,+,3)(1)$ and $(4,-,3)(6)$.
\item[(15)]
$F=\{\{16\},\{12346\}\}$\\
$F$ satisfies $(4,+,2)(6)$ and $(1,-,2)(1)$, so $C_2=\{2,7\}$ and hence $C_7=\{2,7\}$ but $F$ satisfies $(1,+,5)(1)$ and $(4,-,5)(6)$, so $C_5=\{5,7\}$ and hence $C_7=\{5,7\}=\{2,7\}$, a contradiction.
\item[(16)]
$F=\{\{23\},\{12346\}\}$\\
$F$ satisfies $(1,+,5)(1)$ and $(4,-,5)(7)$, so $C_5=\{5,7\}$. $F$ satisfies $(4,-,4)(6)$, so $7,1,2,3\notin C_4$ and since $C_5=\{5,7\}$, we have $5\notin C_4$ (if not, $7\in C_5\subseteq C_4$). Hence $C_4=\{4,6\}$, a contradiction since $C_5$ and $C_4$ are crossing.
\item[(17)]
$F=\{\{12\},\{12356\}\}$\\
$F$ satisfies $(4,+,4)(6)$ and $(1,-,4)(1)$, so $C_4=\{2,4\}$ and hence $C_2=\{2,4\}$ but $F$ satisfies $(1,+,7)(1)$ and $(4,-,7)(6)$, so $C_7=\{2,7\}$ and hence $C_2=\{2,7\}=\{2,4\}$, a contradiction.
\item[(18)]
$F=\{\{16\},\{12356\}\}$\\
$F$ satisfies $(2,+,4)(2)$ and $(4,-,4)(5)$.
\item[(19)]
$F=\{\{56\},\{12356\}\}$\\
$F$ satisfies $(4,+,4)(6)$ and $(1,-,4)(1)$, so $C_4=\{2,4\}$ and hence $C_2=\{2,4\}$ but $F$ satisfies $(1,+,7)(1)$ and $(4,-,7)(6)$, so $C_7=\{2,7\}$ and hence $C_2=\{2,7\}=\{2,4\}$, a contradiction.
\item[(20)]
$F=\{\{56\},\{1234,56\}\}$ or $\{\{57\},\{1234,57\}\}$\\
$F$ satisfies $(4,+,2)(1)$ and $(2,-,2)(1)$.
\item[(21)]
$F=\{\{67\},\{1235,67\}$\\
$F$ satisfies $(4,+,2)(3)$ and $(2,-,2)(1)$.
\item[(22)]
$F=\{\{67\},\{1245,67\}\}$\\
$F$ satisfies $(3,+,3)(3)$ and $(3,-,3)(3)$.
\item[(23)]
$F=\{\{13\},\{123456\}\}$\\
$F$ satisfies $(4,+,5)(3)$ and $(3,-,5)(1)$.
\item[(24)]
$F=\{\{16\},\{123456\}\}$\\
$F$ satisfies $(3,+,3)(1)$ and $(3,-,3)(1)$.
\item[(25)]
$F=\{\{24\},\{123456\}\}$\\
$F$ satisfies $(2,+,7)(2)$ and $(4,-,7)(5)$.
\item[(26)]
$F=\{\{26\},\{123456\}\}$\\
$F$ satisfies $(3,+,4)(1)$ and $(3,-,4)(1)$.
\item[(27)]
$F=\{\{12\},\{12345,67\}\}$\\
$F$ satisfies $(3,+,3)(1)$ and $(3,-,3)(3)$.
\item[(28)]
$F=\{\{15\},\{12345,67\}\}$\\
$F$ satisfies $(3,+,3)(1)$ and $(3,-,3)(1)$.
\item[(29)]
$F=\{\{23\},\{12345,67\}\}$\\
$F$ satisfies $(2,+,4)(1)$ and $(4,-,4)(6)$.
\item[(30)]
$F=\{\{67\},\{12345,67\}\}$\\
$F$ satisfies $(3,+,3)(1)$ and $(3,-,3)(1)$.
\item[(31)]
$F=\{\{56\},\{1234,567\}\}$ or $\{\{57\},\{1234,567\}\}$\\
$F$ satisfies $(4,+,2)(1)$ and $(2,-,2)(1)$.
\item[(32)]
$F=\{\{124\},\{1234\}\}$, $\{\{13\},\{123\},\{1234\}\}$, $\{\{14\},\{124\},\{1234\}\}$ or $\{\{24\},\{124\},\{1234\}\}$\\
$F$ satisfies $(3,+,5)(1)$ and $(3,-,5)(1)$.
\item[(33)]
$F=\{\{12\},\{126\},\{123456\}\}$ or $\{\{16\},\{126\},\{123456\}\}$\\
$F$ satisfies $(3,+,4)(1)$ and $(3,-,4)(1)$.
\item[(34)]
$F=\{\{12\},\{12,34\},\{123456\}\}$\\
$F$ satisfies $(3,+,7)(3)$ and $(3,-,7)(3)$.
\item[(35)]
$F=\{\{12\},\{12,45\},\{123456\}\}$\\
$F$ satisfies $(3,+,3)(3)$ and $(3,-,3)(3)$.
\item[(36)]
$F=\{\{12\},\{12,56\},\{123456\}\}$\\
$F$ satisfies $(3,+,3)(1)$ and $(3,-,3)(3)$.
\item[(37)]
$F=\{\{16\},\{16,23\},\{123456\}\}$\\
$F$ satisfies $(3,+,4)(1)$ and $(3,-,4)(3)$.
\item[(38)]
$F=\{\{16\},\{16,34\},\{123456\}\}$\\
$F$ satisfies $(4,+,2)(6)$ and $(1,-,2)(1)$, so $C_2=\{2,7\}$ and hence $C_7=\{2,7\}$ but $F$ satisfies $(1,+,5)(1)$ and $(4,-,5)(6)$, so $C_5=\{5,7\}$ and hence $C_7=\{5,7\}$, a contradiction.
\item[(39)]
$F=\{\{23\},\{23,45\},\{123456\}\}$\\
$F$ satisfies $(4,+,7)(7)$ and $(4,-,7)(7)$.
\end{itemize}
So we are done.
\end{proof}

\section{Proof that $B_7$ is CAT(0)}

By the construction in \cite{Brad01}, $K_n$ is isometric to $|NCP_n|$. So, to prove $B_n$ is CAT(0), it suffices to prove that $X:=\text{lk}(e_{01},|NCP_n|)$ is CAT(1). In \cite{TDP13}, proof that $X$ is CAT(1) uses an embedding of $X$ to a spherical building. We will describe their proof.\\
Define $L_n$ to be the poset of all linear subspaces of $\{(x_1,\cdots,x_n)\in \mathbb{R}^n | \sum_{i=1}^nx_i=0\}$ ordered by inclusion. Denote $|L_n|$ be the geometric realization of $L_n$ with the orthoscheme metric. For an element $A\in L_n$ define its \textit{rank} rk $A$= dim $A$. There is a rank preserving injective poset map $\phi$ from $NCP_n$ to $L_n$. $\phi$ sends $P\in NCP_n$ to $\phi(P)=\{(x_1,\cdots,x_n)\in\mathbb{F}^n | \sum_{j=1}^{k}x_{i_j}=0\text{ if }\{x_{i_1},\cdots,x_{i_k}\}\in P\}\in L_n$. For convenience, for a poset, we denote 0 be the minimun element and 1 be the maximum element. Recall that $X=\text{lk}(e_{01},|NCP_n|)$ and let $B=\text{lk}(e_{01},|L_n|)$. Since every vertex of $X$ correspond to a vertex in $NCP_n\setminus\{0,1\}$, we may think $X$ as a geometric realization of the poset $\mathcal{NCP}_n=NCP_n\setminus\{0,1\}$. Similarly, we may think $B$ as a geometric realization of $L_n\setminus \{0,1\}$. $\phi|_{\mathcal{NCP}_n}$ induces an injective simplicial map $i:X\rightarrow B$. So we can think $X$ as a subcomplex of $B$. It is proved in \cite{MA99} that $B$ is a CAT(1) space (as a spherical building).\\
Using this embedding of $X$ to the CAT(1) space $B$, if $X$ is locally CAT(1) (see the next paragraph), then by results of Bowditch \cite{Bowd95}, there exists a short (meaning that length is less than $2\pi$) local geodesic in $X$ which is not a local geodesic in $B$. From this observation, we define following notions. For a local geodesic path or loop $l:D\rightarrow X$ ($D=I$ or $S^1$), we say $t\in D$ is a \textit{turning point} if $i\circ l:D\rightarrow B$ is not a local geodesic in $t$. If $l$ is a local geodesic loop in $X$ and $t$ is a turning point, then we say a face $F$ in $X$ is a \textit{turning face} if it is the smallest (with respect to inclusion) intersection of a chamber containing $l([t,t+\epsilon))$ and a chamber containing $l((t-\epsilon,t])$ for some $\epsilon>0$. Here are some definitions used in lemmas about turning faces in \cite{TDP13}. We say that a face is \textit{universal} if all of its vertices have exactly one block that is consecutive. We say a point is \textit{universal} if it is a point in a universal face. We say $x,y\in X$ \textit{fail modularity} if $x\vee y\notin X$ or $x\wedge y\notin X$, where $\vee,\wedge$ is the join and meet in $B$. We will state several results about turning faces proved in \cite{TDP13}.\\
Our next goal is to prove that if $\text{lk}(e_{01},|NCP_m|)$ is CAT(1) for all $3\leq m<n$, then $\text{lk}(e_{01},|NCP_n|)$ is locally CAT(1). By a result in \cite{TJ10}, $\text{lk}(e_{01},|NCP_n|)$ is locally CAT(1) iff for any nonempty cell $\sigma$ of $\text{lk}(e_{01},|NCP_n|)$, we have
\[\text{lk}(\sigma,\text{lk}(e_{01},|NCP_n|))\simeq\text{lk}(\sigma',|NCP_n|)\]
is not CAT(1), where $\sigma'$ is a cell in $|NCP_n|$ containing $e_{01}$. Since $|NCP_n|$ has the orthoscheme metric, by a result of \cite{TJ10} $\text{lk}(\sigma',|NCP_n|)$ is a spherical product of diagonal link of $|NCP_m|$'s for $3\leq m<n$. Since each factor is CAT(1) by assumption, $\text{lk}(\sigma',|NCP_n|)$ is CAT(1) as desired. From now on, for a fixed $n\geq 3$, we assume $\text{lk}(e_{01},|NCP_m|)$ is CAT(1) for all $3\leq m<n$ and hence $\text{lk}(e_{01},|NCP_n|)$ is locally CAT(1).\\

We say a rectifiable loop $l$ in a complete locally compact path-metric space is \textit{shrinkable} if there is a rectifiable loop $l'$ of length shorter than the length of $l$ and there is a homotopy between $l$ and $l'$ going through length non-incresing rectifiable loops. If $l$ is not shrinkable, then we say $l$ is \textit{unshrinkable}. Now we will state lemmas in \cite{TDP13} which give relationship between CAT(1)-ness of $X$ and conditions in section 1.

\begin{lem}[\cite{TDP13},Lemma 3.12]\label{lem7} Let l be a locally geodesic loop in X. Then for every turning point of l in B, its turning face has a corank which contains two consecutive integers.
\end{lem}

\begin{lem}[\cite{TDP13},Lemma 3.29]\label{lem8} Let $x\in X$ be a universal point, and let l be a short loop in X through x. Then l is shrinkable in X.
\end{lem}

\begin{lem}[\cite{TDP13},Lemma 3.15]\label{lem9} Let $l:I\rightarrow X$ be a locally geodesic segment in X with a turning point t in B. Let $E^+$ (respectively $E^-$) be minimal faces in X containing the image under l of a right (respectively left) $\epsilon$-neighbourhood of t for some $\epsilon>0$. Then there exist vertices $x^+\in E^+$ and $x^-\in E^-$ which fail modularity.
\end{lem}

We will identify every face of $X$ with the corresponding chain in $\mathcal{NCP}_n$. Lemma \ref{lem7} tells us that if $F$ is a turning face, then $F$ satisfies condition I. Lemma \ref{lem8} tells us that if $l$ is an unshrinkable short loop and $F$ is its turning face, then $F$ is not a universal face and hence $F$ satisfies condition II. Now we will give a relationship between Lemma \ref{lem9} and condition III.\\

\begin{thm}\label{thm10} If $F$ is a turning face of $X$, then $F$ satisfies condition III.
\end{thm}

\begin{proof} We will identify vertices of $X$ with corresponding elements of $\mathcal{NCP}_n$. If $F=\{x_1,\cdots,x_s\}$ is a turning face, then by Lemma \ref{lem9}, there is $x^+\in E^+$ and $x^-\in E^-$ which fail modularity ($E^+, E^-$ are as in the Lemma \ref{lem9}). The fact $E^+$, $E^-$ are faces containing $F$ impiles that $x^+$, $x^-$ are adjacent to every vertex of $F$ in $X$. Since $x^+, x^-$ fail modularity, we cannot have $x^+\leq x^-$ or $x^+\geq x^-$. Therefore, either $x_i<x^+$,$x^-<x_{i+1}$ for some $1\leq i\leq k-1$, or $x^+$,$x^-<x_1$, or $x_s<x^+$,$x^-$.\\
\textbf{Case 1)} $x^+\vee x^-\notin X$.\\
Let $e_i$ be the $i$-th standard generator of $\mathbb{F}^n$. As an element of $L_n\setminus\{0,1\}$, $x^+$ is generated by $\{e_{i_1}-e_{i_{j}}|\{i_1,\cdots,i_{i_k}\}\in P^+,2\leq j\leq i_k\}$ and $x^-$ is generated by $\{e_{i_1}-e_{i_{j}}|\{i_1,\cdots,i_{i_k}\}\in x^-,2\leq j\leq i_k\}$. $x^+\vee x^-$ is generated by both of them and so it is easy to see that $x^+\vee x^-$ is generated by $\{e_{i_1}-e_{i_{j}}|\{i_1,\cdots,i_{i_k}\}\in x^+\vee_{\mathcal P_n}x^-,2\leq j\leq i_k\}$, where $\vee_{\mathcal P_n}$ is the join in $\mathcal P_n$. Hence $x^+\vee_{\mathcal P_n}x^-\notin \mathcal{NCP}_n$, which means that $x^+\nparallel x^-$.\\
\textbf{Case 2)} There is a non-zero vector $v=(v_1,\cdots,v_n)\in (x^+\wedge x^-)\setminus (x^+\wedge_{\mathcal{NCP}_n}x^-)$. $v\in x^+\wedge x^-$ means that $\sum_{i\in\sigma}v_i=0$ if $\sigma\in x^+\cup x^-$. $0\neq v\notin x^+\wedge_{\mathcal{NCP}_n}x^-$ means that $^\exists\sigma_1\in x^+$, $^\exists\tau_1\in x^-$ such that $\sigma_1\cap\tau_1\neq\emptyset$ and $\sum_{i\in\sigma_1\cap\tau_1}v_i\neq0$. Clearly, $\sigma_1\not\subseteq\tau_1$ and $\tau_1\not\subseteq\sigma_1$ by above conditions. There exists $\sigma_2\in x^+$ such that $\sigma_2\neq\sigma_1$ and $\sigma_2\cap\tau_1\neq\emptyset$ because if not, for each $i\in\tau_1\setminus\sigma_1\neq\emptyset$, we have $\{i\}\in x^+$, so $v_i=0$. Then $\sum_{i\in\sigma_1}v_i\neq0$, contradiction. By the same reason, actually we can choose $\sigma_2$ satisfying an additional condition  $\sum_{i\in\sigma_2\cap\tau_1}v_i\neq0$. In the same way, we choose $\tau_2$, $\sigma_3$ and so on until the same element appears. So we get a set of distinct blocks $\sigma_1,\cdots,\sigma_k$ and $\tau_1,\cdots,\tau_k$, $k\geq 2$ satisfying $\sigma_i\cap\tau_i\neq\emptyset$ and $\sigma_{i+1}\cap\tau_i\neq\emptyset$. Now, consider 4 distinct non-empty sets $\sigma_1\cap\tau_1$, $\sigma_2\cap\tau_1$, $\sigma_2\cap\tau_2$, $\sigma_{2+1}\cap\tau_2$ (+ is done modulo $k$). Let $y^+ =(x^+ \setminus \{\sigma_1,\sigma_2\}) \cup\{\tau_k\cap\sigma_1,\tau_2\cap\sigma_{2+1},(\sigma_1\cap\tau_1)\cup(\sigma_2\cap\tau_2)\}$ and $y^- =(x^- \setminus \{\tau_1,\tau_2\}) \cup\{\sigma_1\cap\tau_1,\sigma_2\cap\tau_2,(\tau_1\cap\sigma_2)\cup(\tau_2\cap\sigma_{2+1})\}$. Then $x^+\wedge_{\mathcal{NCP}_n}x^-\leq y^+,y^-\leq x^+\vee_{\mathcal{NCP}_n}x^-$, so $y^+$ and $y^-$ satisfy (i) of condition III as well. Moreover, $(\sigma_1\cap\tau_1)\cup(\sigma_2\cap\tau_2)\in y^+$ and $(\tau_1\cap\sigma_2)\cup(\tau_2\cap\sigma_{2+1})\in y^-$ such that $((\sigma_1\cap\tau_1)\cup(\sigma_2\cap\tau_2))\nparallel((\tau_1\cap\sigma_2)\cup(\tau_2\cap\sigma_{2+1}))$ as desired.
\end{proof}

We say that a forest in $\mathbb{C}$ is a \emph{non crossing forest} if it is an embedded forest in $\mathbb{C}$ with vertices are elements of $U_n$ and edges are straight line segments and $1\leq$(the number of edges)$\leq n-2$ . Each non crossing forest corresponds to a vertex of $X$ via the correspondence that $i,j$ are in the same block iff they are in the same connected components in the non crossing forest.\\
We say that a tree in $\mathbb{C}$ is a \emph{non crossing spanning tree} if it is an embedded tree in $\mathbb{C}$ with vertices are elements of $U_n$ and edges are straight line segments and the number of edges is $n-1$. For each non crossing spanning tree, we associate a subcomplex $A$ of $X$ called \emph{apartment} that is geometrically realized by all vertices correspond to non crossing forests which is a subset of the non crossing spanning tree. It is known that apartments are isometric to $S^{n-3}$.\\

The following is defined in \cite{TDP13}. We say that a vertex $u$ of a face $F$ is \emph{dominant} if for all apartments $A$ containing $u$, $A$ contains $F$. Note that dominant vertex of a face is unique if exists, in fact, if $u_1<u_2$ are dominant vertices of $F$, then if a block $b_2$ of $u_2$ properly contains a block $b_1$ of $u_1$, then we can construct an apartment containing $u_2$ such that the corresponding non crossing spanning tree connects all vertices of $b_2$ but does not connects all vertices of $b_1$ by using vertices of $b_1$ only. Then $A$ does not contain $F$, a contradiction. If every block of $u_1$ is some block of $u_2$, then there must be a block $b_2$ such that each elements are singletons in $u_1$. We can construct an apartment $A$ containing $u_1$ such that the corresponding non crossing spanning tree connects all vertex of $b_1$ to a vertex not in $b_1$. Then $A$ does not contain $F$, a contradiction.\\

Recall that $F_i'=F_i\cap\{i-1,i+1\}$. It is true that if $u$ is the dominant vertex of $F$, then for all $1\leq i \leq n$, we have $u_i'=F_i'$. Indeed, $u_i\supseteq F_i$ by definition, so $u_i'\supseteq F_i'$. Now we will show $u_i'\subseteq F_i'$. If $u_i$ is a block of $u$, then $F_i$ is also a block of a vertex of $F$. If $u_i$ properly contains $F_i$, then we can construct an apartment containing $u_i$ but not $F$ using the same argument in the last paragraph. Hence $u_i=F_i$ and so $u_i'=F_i'$. The remaining case is that $u_i=U_n$. In this case, if $u_i'\neq F_i'$, then say $i+1\notin F_i'$. Then we construct an apartment $A$ containing $u$ corresponding the non crossing spanning tree that containing the edge connecting $i$,$i+1$ and $i$ is not connected with any other vertices. This can be done because $i$ is a singleton in $u$. Then $A$ does not contain $F$ because $F_i$ is a block of a vertex of $F$ and does not containing $i+1$ and all sub non crossing forest of the no crossing spanning tree cannot connect $i$ with other vertices of $F_i$ by a path avoding $i+1$.\\

To prove that turning faces have to satisfy condition IV, we first state a lemma proved in \cite{TDP13}.\\

\begin{lem}[\cite{TDP13},Lemma 3.35]\label{lem10} Let C be a chamber in X and $i\in U_n$, and let i,j be consecutive elements of $U_n$. If $C_i$ contains j, then there exists an apartment in X containing C, v and w, where v is the boundary edge having the single nontrivial block $\{i,j\}$ and w is the universal vertex opposite to v in B having the single nontrivial block $\{1,\cdots,\hat{i},\cdots,n\}$.\\
\end{lem}

In fact, this lemma can be strengthened so that $C$ need not be a chamber and can be any faces.\\

\begin{lem}\label{lem11} Let F be a face in X and $i\in U_n$, and let i,j be consecutive elements of $U_n$. If $F_i$ contains j, then there exists an apartment in X containing F, v and w, where v is the boundary edge having the single nontrivial block $\{i,j\}$ and w is the universal vertex opposite to v in B having the single nontrivial block $\{1,\cdots,\hat{i},\cdots,n\}$.
\end{lem}

\begin{proof} Let $F$ be a face in $X$ and $i\in U_n$, and let $i,j$ be consecutive elements of $U_n$. To prove the lemma, we will find a chamber $C$ containing $F$ such that $j\in C_i$. Then by the lemma \ref{lem10}, we get an apartment $A$ containing $C$, $v$ and $w$ and hence $A$ also contains $F$, $v$ and $w$.\\
A chamber $C$ is constructed as follows. Let $F=\{x_{k_1},\cdots,x_{k_s}\}$, where indices are equal to the rank of corresponding vertex. If $F_i=U_n$, then all vertices of $F$ do not have a block containing $i$, hence $x_{k_s}\leq \{\{1,\cdots,\hat{i},\cdots,n\}\}$. If $F_i$ is a block of a vertex $x_{k_t}=\{a_1,\cdots,a_r\}$ with the minimal $k_t$. Say $F_i=a_p$, $1\leq p \leq r$. For cases where $k_t=1$ or $k_{t-1}=k_t-1$, take any chamber $C$ containing $F$, then we have $j\in C_i$. For the remaining cases, let $v_{k_t-1}=\{a_1,\cdots,a_{p-1},a_p\setminus\{i\},a_{p+1},\cdots,a_r\}.$ Then we have $v_{k_{t-1}}<v_{k_t-1}<v_{k_t}$ (or $v_{k_t-1}<v_{k_t}$ for $t=1$ case) by minimality of $t$. Now, any chamber $C$ containing $F$ and the vertex $v_{k_t-1}$ have the property that $j\in C_i$.
\end{proof}

Now, we use this lemma to give a relationship between turning face and condition IV. This is essentially the same routine of [\cite{TDP13},Lemma 3.36] and [\cite{TDP13},Lemma 3.37].\\

\begin{thm}\label{thm11} If $F$ is face of $X$ and $F$ does not satisfy condition VI, then there exist universal opposite vertices $v$, $w$ and two apartments $A$, $A'$ such that
\begin{itemize}
\item[(1)] $F,v,w$ are contained in $A$
\item[(2)] $C,v,w$ are contained in $A'$
\end{itemize}
Consequently, $F$ is not a turning face.
\end{thm}

\begin{proof} Recall condition IV is that there is a maximal chain $C$ in $\mathcal{NCP}_n$ such that $F_i'\cap C_i'=\emptyset$ for all $1\leq i \leq n$.\\
If $F$ does not satisfy condition IV, then there exists a chamber $C$ in $X$ such that for some consecutive integers $i,j$, we have $j\in F_i\cap C_i$. Therefore, by lemma \ref{lem11}, applied to $F$ and $C$ independently, there are apartments $A$, containing $F,v,w$, and $A$, containing $C,v,w$. Suppose $F$ is a turning face and let $l$ be a short unshrinkable loop in $X$ passing through $F$, then the proof of [\cite{TDP13},Lemma 3.37] says that $l$ is shrinkable. Therefore $F$ is not a turning face.
\end{proof}

\begin{rem} In \cite{TDP13}, their version of theorem \ref{thm11} is obtained by changing condition IV to the condition, say condition IV$'$, that either $F$ does not have a dominant vertex of if $u$ is the dominant vertex of $F$, then there is a maximal chain $C$ in $\mathcal{NCP}_n$ such that $u_i'\cap C_i'=\emptyset$ for all $1\leq i \leq n$. Note that since $u_i'=F_i'$, condition IV$'$ and condition IV are equivalent if $F$ has the dominant vertex. Hence, condition IV is stronger than condition IV$'$. For $n=7$, there is a face $F$ satisfying condition I,II,III and IV$'$, for example $F=\{\{1,3\},\{1,3,4,5,7\}\}$ satisfies condition I,II,III and does not have dominant vertex so it satisfies condition IV$'$.
\end{rem}

In summary, we have the following theorem.

\begin{thm}\label{thm12} If $\text{lk}(e_{01},|NCP_n|)$ is not CAT(1) and $\text{lk}(e_{01},|NCP_m|)$ is CAT(1) for all $m<n$, then there is a chain $F$ in $\mathcal{NCP}_n$ such that $F$ and $F^*$ satisfies conditions I,II,III and IV.
\end{thm}

\begin{proof} Since $X=\text{lk}(e_{01},|NCP_n|)$ is locally CAT(1) but not CAT(1), there is an unshrinkable short local geodesic loop $l$ in $X=\text{lk}(e_{01},|NCP_n|)$ by \cite{Bowd95}. Since $B$ is CAT(1) and $i\circ l$ is short, $i\circ l$ is not a local geodesic loop in $B$, where $i$ is the embedding from $X$ to $B$. Hence there is a turning point $t$ and let $F$ be its turning face. $F$ has to satisfy condition I,II,III and IV, by Theorem \ref{thm10}, Theorem \ref{thm11} and explanations below Lemma \ref{lem9}. Consider the isometry $\phi:X\rightarrow X$ induced by the order reversing bijection of $\mathcal{NCP}_n$. The unshrinkable short local geodesic loop $\phi\circ i$ has turning point $t$ again . So, its turning face is  $F^*$ and it satisfies condition I,II,III and IV as well.
\end{proof}

\begin{cor}\label{cor12} For $n\leq 7$, the $n$-strand braid group $B_n$ is a CAT(0) group.
\end{cor}

\begin{proof} It is proved in \cite{TDP13} that $\text{lk}(e_{01},|NCP_m|)$ is CAT(1) and hence $B_m$ is CAT(0) for all $m\leq6$. If $\text{lk}(e_{01},|NCP_n|)$ is not CAT(1), then there is a chain $F$ in $\mathcal{NCP}_n$ such that $F$ and $F^*$ satisfies conditions I,II,III and IV, by Theorem \ref{thm11}. But there does not exist such chain by Theorem \ref{thm5}, a contradiction occurs. Hence $\text{lk}(e_{01},|NCP_7|)$ is CAT(1) and $B_7$ is a CAT(0) group.
\end{proof}



\end{document}